\documentclass[a4paper,11pt,reqno]{amsart}

\usepackage{amssymb}
\usepackage{epsfig}
\usepackage{amsfonts}
\usepackage{amsmath}
\usepackage{euscript}
\usepackage{amscd}
\usepackage{amsthm}
\usepackage{enumitem}
\DeclareMathAlphabet{\mathpzc}{OT1}{pzc}{m}{it}
\usepackage{enumitem}
\usepackage[numbers,sort&compress]{natbib}
\usepackage{color}
\usepackage{mathtools}
\usepackage[hypertexnames=false, colorlinks, citecolor=red, linkcolor=blue, urlcolor=red]{hyperref}

\usepackage{marginnote}

\usepackage[normalem]{ulem}

\newcommand{\marginextend}[1]{ \addtolength{\oddsidemargin}{-#1}  \addtolength{\evensidemargin}{-#1}
	\addtolength{\textwidth}{#1}\addtolength{\textwidth}{#1}}
\newcommand{\updownextend}[1]{ \addtolength{\topmargin}{-#1}  \addtolength{\textheight}{#1}
	\addtolength{\textheight}{#1}}
\marginextend{1.5cm}
\updownextend{0cm}
\allowdisplaybreaks[4]

\usepackage{pst-node}
\usepackage{tikz-cd}
\usepackage{mathrsfs}
\DeclareFontFamily{OT1}{pzc}{}
\DeclareFontShape{OT1}{pzc}{m}{it}{<-> s * [1.10] pzcmi7t}{}
\DeclareMathAlphabet{\mathpzc}{OT1}{pzc}{m}{it}

\DeclareSymbolFont{SY}{U}{psy}{m}{n}
\DeclareMathSymbol{\emptyset}{\mathord}{SY}{'306}

\theoremstyle{plain}

\newtheorem{thm}{Theorem}[section]
\newtheorem*{thm*}{Theorem}
\newtheorem{cor}[thm]{Corollary}
\newtheorem{lem}[thm]{Lemma}
\newtheorem{prop}[thm]{Proposition}
\newtheorem{defn}[thm]{Definition}
\newtheorem{rem}[thm]{Remark}
\newtheorem{ex}[thm]{Example}

\newtheorem{observation}[equation]{Observation}

\newtheoremstyle{named}{}{}{\itshape}{}{\bfseries}{.}{.5em}{#1 \thmnote{#3}}
\theoremstyle{named}

\numberwithin{equation}{section}

\def\N{\mathbb{N}}

\def\beq{\begin{eqnarray}}
	\def\eeq{\end{eqnarray}}
\def\beqa{\begin{eqnarray*}}
	\def\eeqa{\end{eqnarray*}}

\newcommand{\be}{\begin{equation}}
	\newcommand{\ee}{\end{equation}}
\newcommand{\bea}{\begin{eqnarray}}
	\newcommand{\eea}{\end{eqnarray}}
\newcommand{\Bea}{\begin{eqnarray*}}
	\newcommand{\Eea}{\end{eqnarray*}}

\newcounter{cnt1}
\newcounter{cnt2}
\newcounter{cnt3}
\newcommand{\blr}{\begin{list}{$($\roman{cnt1}$)$}
		{\usecounter{cnt1} \setlength{\topsep}{0pt}
			\setlength{\itemsep}{0pt}}}
	\newcommand{\bla}{\begin{list}{$($\alph{cnt2}$)$}
			{\usecounter{cnt2} \setlength{\topsep}{0pt}
				\setlength{\itemsep}{0pt}}}
		\newcommand{\bln}{\begin{list}{$($\arabic{cnt3}$)$}
				{\usecounter{cnt3} \setlength{\topsep}{0pt}
					\setlength{\itemsep}{0pt}}}
			\newcommand{\el}{\end{list}}
		
		\newcommand{\hnorm}[1]{\left\lVert#1\right\rVert}

		\vspace{5mm}

		\title{Metrics on $C^{\ast}$-algebras of étale groupoids from length functions }
		\author[A. Chattopadhyay]{Arnab Chattopadhyay}
		
		\author[Md A. Hossain]{Md Amir Hossain}
		
		\author[S. Joardar]{Soumalya Joardar}

        \address[A. Chattopadhyay]{Indian Institute of Science Education And Research Kolkata, Mohanpur 741246, Nadia, West Bengal, India} \email{ac23rs002@iiserkol.ac.in}
        
        \address[Md A. Hossain]{The Institute of Mathematical Sciences, A CI of Homi Bhabha  National Institute,
			4th Cross Street, CIT Campus, Taramani, Chennai, 600113, India.} \email{mdamirhossain18@gmail.com}

            \address[S. Joardar]{Indian Institute of Science Education And Research Kolkata, Mohanpur 741246, Nadia, West Bengal, India} \email{soumalya@iiserkol.ac.in}

\begin{document}
			\begin{abstract}	   
		We show that for an étale groupoid with compact unit space, the natural Dirac type operator from a continuous length function produces a natural pseudo-metric on the state space of the corresponding reduced $C^{\ast}$-algebra. For a transformation groupoid with a continuous, proper length function with rapid decay, the state space decomposes into genuine metric spaces with a uniform finite diameter fibred over the state space of the compact unit space. Moreover, when the unit space of the transformation groupoid has finitely many points, the metric on each fibre metrizes the weak$^\ast$-topology. 
			\end{abstract}
			\maketitle
			\section{Introduction}
			One of the key aspects of the spectral triple formalism of noncommutative geometry pioneered by Alan Connes is that it encodes the metric data of the $C^{\ast}$-algebra of the triple. The motivation behind the idea is that the geodesic distance between two points of a compact Riemannian spin manifold can be recovered from its associated spectral triple. Recall that a general spectral triple is a data $(\mathcal{A},\mathcal{H}, D)$ where $\mathcal A$ is a $\ast$-subalgebra of $B(\mathcal H)$ and $D$ is some unbounded operator acting on a Hilbert space $\mathcal H$ such that $[D,a]\in B(\mathcal{H})$ for all $a\in\mathcal{A}$. Then the state space of the $C^{\ast}$-completion of $\mathcal{A}$ can be equipped with the following distance (see \cite{Connes-1989-Compact-metric-sp-Fred-mod-Hyperfinteness}) 
			\begin{displaymath}
				d(\phi,\psi):={\rm sup} \{|\phi(a)-\psi(a)| : \hnorm{[D,a]}\leq 1\},
			\end{displaymath}
			where $\phi$ and $\psi$ are states on the $C^{\ast}$-completion of $\mathcal A.$ The above formula can be generalized in the sense that if a $C^{\ast}$-algebra $A$ is equipped with a seminorm $L$, then a similar formula as above can be given for a distance between two states where $\hnorm{[D,a]}$ is replaced by $L(a)$ in the above formula. In fact, for a classical compact metric space $(X,d)$ one can define a seminorm $L$ on $C(X)$ by the formula 
			\begin{displaymath}
				L (f):={\sup\limits}_{x\neq y}\frac{|f(x)-f(y)|}{d(x,y)}.  
			\end{displaymath}
			This seminorm recovers the metric data of the space $X$.
			With this observation, Marc Rieffel defined a compact quantum metric space as a pair $(A,L)$ where \(A\) is a \(C^*\)-algebra and $L$ is a densely defined Lip-norm (see~\cite[Definition 2.2]{Article}) such that the corresponding metric induces the weak$^\ast$-topology on the state space $\mathcal{S}(A)$. Then a number of works followed to find concrete examples of such compact quantum metric spaces over the last few decades (see \cite{Ozawa_Rieffel_2005, Rieffel-2002-Group-C-alg-as-compact-quan-metr-sp, Connes-1989-Compact-metric-sp-Fred-mod-Hyperfinteness, Rieffel-1999-Metrics-on-state-spaces, Long-Wu-2017-Twisted-group-C-alg-as-CQMS, Rieffel-1998-Metrics-on-state-from-action-of-cmpt-gp, Article, DKJK, Latrmolire2018TheGP, FL, Quaegebeur2010IsometricCO}). The quantum metric space structure has also been useful in defining and studying the quantum analogue of Gromov-Hausdorff distances between compact quantum metric spaces (see \cite{Latrmolire2018TheGP, Article}). In fact, this was one of the motivations to define compact quantum metric spaces.  Apart from finding examples of compact quantum metric spaces, a notion of quantum group of isometries with respect to the compact quantum metric space structure has also been studied in \cite{Quaegebeur2010IsometricCO}.\\
			\indent A rich supply of examples of compact quantum metric spaces has come from the reduced group $C^{\ast}$-algebras (see \cite{Long-Wu-2017-Twisted-group-C-alg-as-CQMS, Rieffel-1998-Metrics-on-state-from-action-of-cmpt-gp, Ozawa_Rieffel_2005, Rieffel-2002-Group-C-alg-as-compact-quan-metr-sp} etc). The Lip-norms mostly come from the natural spectral triples from the length functions of the underlying groups. Recently, examples have been constructed on the twisted reduced group $C^{\ast}$-algebras equipped with proper length functions with the rapid decay property. The key idea is to define a sequence of Lipschitz-seminorms $L^{k}$ arising out of the spectral triple induced by the length function and to show that there is a large enough $k$ such that the Lip-seminorms become Lip-norms. \\
			\indent Now, with the introduction of a length function with the rapid decay property on groupoids (see \cite{Weygandt-2024-Rapid-decay-for-etale-gpd}), it is natural to seek the extensions of the results on the reduced group $C^{\ast}$-algebras to the realm of reduced groupoid $C^{\ast}$-algebras and twisted versions of them. We set out to explore the idea in this paper. In order to work in the unital set-up, we restrict our attention to the reduced $C^{\ast}$-algebras of étale groupoids with compact unit spaces such that the groupoids have proper length functions with the rapid decay property. We mostly follow the footsteps of \cite{Long-Wu-2017-Twisted-group-C-alg-as-CQMS, Rieffel-1998-Metrics-on-state-from-action-of-cmpt-gp}. However, the obvious approach hits a stumbling block as a natural sequence of seminorms $L^k$ constructed out of the length functions have $C(X)$ as the kernel of \(L^k\), where $X$ is the compact unit space of the groupoid. Then the metric on the state space becomes a pseudo-metric. The state space turns out to be a disjoint union of weak$^\ast$-compact subspaces indexed by the probability measures on the compact unit space $X$. The distance between any two states coming from two different fibres becomes $+\infty$. Therefore, the weak$^\ast$-topology being connected, cannot be metrized by such a metric. Then the relevant question becomes whether the weak$^\ast$-topologies on each fibre can be metrized by the metric or not. We show that when the étale groupoid is a transformation groupoid equipped with a continuous, proper length function with rapid decay property, then each fibre has a uniform finite diameter with the induced metric. Moreover, if the unit space has finitely many points, then the weak$^\ast$-topology of each fibre is metrized by the metric induced by a seminorm $L^k$ for a large enough $k$.

			\section{Main results}
			
			We shall only work with {\'e}tale groupoid. We briefly recall the definitions of groupoids and {\'e}tale groupoids. For details we refer the reader to~\cite{renault2006groupoid,sims2020operator,williams2019tool,brownc}.
			
			A groupoid is a small category \(\mathcal{G}\) where every element has an inverse. The set of
			units in the groupoid \(\mathcal{G}\) is denoted by \(\mathcal{G}^{(0)}\), and the set of all composable pairs is
			denoted by \(\mathcal{G}^{(2)}\). The multiplication operation is defined as a mapping from \(\mathcal{G}^{(2)}\) to
			\(\mathcal{G}\) and the inverse is defined from \(\mathcal{G}\to \mathcal{G}\). 
			Additionally, there are two functions: the range map \(r\) and the source map \(s\), both of which are defined as \(r, s \colon \mathcal{G} \to \mathcal{G}^{(0)}\). Specifically, these maps are given by \(r(\gamma) = \gamma \gamma^{-1}\) and \(s(\gamma) = \gamma^{-1} \gamma\) for \(\gamma \in \mathcal{G}\).
			For a given unit \(x \in \mathcal{G}^{(0)}\), we define the sets \(\mathcal{G}^x = r^{-1}(x)\) and \(\mathcal{G}_x = s^{-1}(x)\), referred to as the range and source fibres, respectively.
			A groupoid is called topological if it has a topology such that both the multiplication
			and inversion maps are continuous with respect to this topology. We call \(\mathcal{G}\)
			a locally compact groupoid if its topology is locally compact and Hausdorff. The
			groupoid \(\mathcal{G}\) is \emph{{\'e}tale} if the range map \(r\) (and hence the source map \(s\)) is a local homeomorphism.
			
			
			\begin{ex}\label{ex2}
				Let $\Gamma$ be a discrete group acting on a locally compact Hausdorff topological space $X$ by homeomorphisms. We consider the transformation groupoid $\mathcal G := \Gamma \ltimes X$. The underlying space of the groupoid is \(\Gamma \times X\). Define the multiplication and inverse maps by $(g, h \cdot x) (h, x) = (gh, x)$ and $(g, x)^{-1} = (g^{-1}, g \cdot x)$ respectively.
				The range map $r$ and the source map $s$ are given by $r(g,x) = g \cdot x$ and $s(g,x) = x$ respectively. The groupoid $\mathcal G$ is étale and the unit space $\mathcal G^{(0)}$ given by $\{e\} \times X : = \{(e,x)\ :\ x \in X\} \simeq X.$
			\end{ex}
			
			Given an {\'e}tale groupoid, one can attach two \(C^*\)-algebras, namely, the full and the reduced groupoid \(C^*\)-algebras. In this paper, we shall be interested in the reduced groupoid $C^{\ast}$-algebras. We briefly recall the construction of groupoid \(C^*\)-algebra. The reader can refer to \cite[Section 5.6]{brownc} and \cite[Section 2]{Weygandt-2024-Rapid-decay-for-etale-gpd} for details. To that end, let $C_c (\mathcal G)$ be the $\ast$-algebra of compactly supported continuous functions on $\mathcal G$ with the multiplication and involution given by 
			\[
			 (f \ast \xi) (\gamma) = \sum\limits_{\beta \in \mathcal G_{s(\gamma)}} f \left (\gamma \beta^{-1} \right ) \xi (\beta) \quad \textup{and} \quad f^{\ast} (\gamma) = \overline {f \left (\gamma^{-1} \right )}
			\]
			for $f, \xi \in C_c (\mathcal G)$ and $\gamma \in \mathcal G$. We define a function $\mathcal E \colon  \mathcal G \to \mathbb C$ by 
			\begin{equation}\label{eq:0} \mathcal E (\gamma) = \begin{cases} 1, \quad \text {if}\ \gamma \in \mathcal G^{(0)}, \\ 0, \quad \text {otherwise.} \end{cases}\end{equation}
			 
			Note that the unit space $\mathcal G^{(0)}$ of an étale groupoid $\mathcal G$ is always clopen (provided $\mathcal G$ is Hausdorff). Hence if $\mathcal G^{(0)}$ is compact then $\mathcal E \in C_c (\mathcal G)$. It is straightforward to verify that $f \ast \mathcal E = \mathcal E \ast f$ for every $f \in C_c (\mathcal G)$, that is, $\mathcal E$ is a multiplicative identity of $C_c (\mathcal G).$ We define a $C_0 (\mathcal G^{(0)})$-valued inner product $\langle \langle \cdot, \cdot \rangle \rangle$ on $C_c (\mathcal G)$ by
			$$\langle \langle \xi, \eta \rangle \rangle (x) = \sum\limits_{\gamma \in \mathcal G_x} \overline {\xi (\gamma)}\ {\eta (\gamma)}$$ 
			for $\xi, \eta \in C_c (\mathcal G)$ and $x \in \mathcal G^{(0)}.$ Let $L^2 (\mathcal G)$ be the Hilbert-$C_0 (\mathcal G^{(0)})$-module obtained by the completion of $C_c (\mathcal G)$ with respect the norm $\| \cdot \|$ given by 
			$$\|f\| : = \biggr ( \sup\limits_{x \in \mathcal G^{(0)}}  \sum\limits_{\gamma \in \mathcal G_x} \left \lvert f(\gamma) \right \rvert^2 \biggr )^{\frac {1} {2}},$$
			 for $f \in C_c (\mathcal G).$ Consider the representation $\lambda\colon C_c (\mathcal G) \to \mathbb B (L^2 (\mathcal G))$ of $C_c (\mathcal G)$ into the space of adjointable operators on $L^2 (\mathcal G)$ given by $\lambda (f) (\xi) = f \ast \xi,$ for $\xi \in C_c (\mathcal G).$ Then it is well-known that $\|\lambda (f)\|_{\text {adj}} \leq \|f\|_I$ for $f \in C_c (\mathcal G),$ where 
			 $$\|f\|_I : = \max \biggr \{\sup\limits_{x \in \mathcal G^{(0)}} \sum\limits_{\gamma \in \mathcal G_x} \left \lvert f (\gamma) \right \rvert,\ \sup\limits_{x \in \mathcal G^{(0)}} \sum\limits_{\gamma \in \mathcal G_x} \left \lvert f(\gamma^{-1}) \right \rvert \biggr \}.$$ 
			 We define the reduced norm $\|\cdot \|_{\text {red}}$ on $C_c (\mathcal G)$ by 
			$$\|f\|_{\text {red}} : = \|\lambda (f) \|_{\text {adj}}.$$ 
			The completion of $C_c (\mathcal G)$ with respect to the reduced norm is known as the reduced groupoid $C^{\ast}$-algebra, which is denoted by $C_r^{\ast} (\mathcal G).$ It is easy to see that $C_c (\mathcal G)$ can be identified with its image in $L^2 (\mathcal G).$
			
			\begin{defn}\label{def-length-fun}\cite[Definition 2.5]{Weygandt-2024-Rapid-decay-for-etale-gpd}
				Let $\mathcal G$ be an {\'e}tale groupoid. By a length function we mean a continuous map $\ell : \mathcal G \to [0, +\infty)$ satisfying the following conditions:
				
				\vspace{2mm}
				
				$(1)$ $\ell (x) = 0$ if and only if $x \in \mathcal G^{(0)},$
				
				$(2)$ $\ell \left (\gamma^{-1} \right ) = \ell (\gamma)$ for any $\gamma \in \mathcal G,$
				
				$(3)$ $\ell (\gamma \eta) \leq \ell (\gamma) + \ell (\eta)$ for any $(\gamma, \eta) \in \mathcal G^{(2)}.$
			\end{defn}
			The length function \(\ell\) is called proper if for every subset $K \subseteq \mathcal G \setminus \mathcal G^{(0)},$ finiteness of the quantity $\sup \{\ell (\gamma) : \gamma \in K\}$ implies that $K$ is pre-compact (see \cite[Definition 2.5]{Weygandt-2024-Rapid-decay-for-etale-gpd}).

			Let $\mathcal G$ be an étale groupoid and let $\ell$ be a length function on $\mathcal G.$ We let $M_{\ell}$ denote the (possibly unbounded) densely defined operator on the Hilbert-$C_0(\mathcal G^{(0)})$-module $L^2 (\mathcal G)$ of pointwise multiplication by $\ell$ on $C_c (\mathcal G)$.
			 Specifically, for \(f \in C_c(\mathcal{G})\) and \(\gamma \in \mathcal{G}\), we have
			 	\(M_{\ell}(f)(\gamma) = \ell(\gamma)f(\gamma). \) The operator \(M^k_{\ell}\) denotes the \(k\)-times composition of \(M_{\ell}\).
		
			\indent Next, we define the derivation associated with the Dirac operator $D = M_{\ell}$ as follows: 
			$$\Delta (f) = \left [D, \lambda (f) \right ] = \left [M_{\ell}, \lambda (f) \right ],$$ for all $f \in C_c (\mathcal G).$ For $k \in \mathbb N,$ we define $\Delta^k (f)$ inductively by 
			$$\Delta^k (f) = \bigl [D, \Delta^{k-1} (f) \bigr ]$$ 
			for all $f \in C_c (\mathcal G).$ For any $f,\xi\in C_c (\mathcal G)$ and $\gamma \in \mathcal G,$ we have
			\Bea
			\Delta (f) (\xi) (\gamma)  & = &  M_{\ell} (\lambda (f) (\xi)) (\gamma) - \lambda (f) (M_{\ell} (\xi)) (\gamma) \\ & = & \ell (\gamma) \lambda (f) (\xi) (\gamma) - (f \ast M_{\ell} (\xi)) (\gamma) \\  & = &  \ell (\gamma) \sum\limits_{\beta \in \mathcal G_{s(\gamma)}} f \left (\gamma \beta^{-1} \right ) \xi (\beta) - \sum\limits_{\beta \in \mathcal G_{s(\gamma)}} f \left (\gamma \beta^{-1} \right ) M_{\ell} (\xi) (\beta) \\ &  = &  \ell (\gamma) \sum\limits_{\beta \in \mathcal G_{s(\gamma)}} f \left (\gamma \beta^{-1} \right ) \xi (\beta) - \sum\limits_{\beta \in \mathcal G_{s(\gamma)}} f \left (\gamma \beta^{-1} \right ) \ell (\beta) \xi (\beta) \\ & = & \sum\limits_{\beta \in \mathcal G_{s(\gamma)}} \left (\ell (\gamma) - \ell (\beta) \right ) f \left (\gamma \beta^{-1} \right ) \xi (\beta).
			\Eea
		
			\begin{lem}
				Let \(f,\xi\in C_c(\mathcal{G})\). Then for \(\gamma \in \mathcal{G}\) and \(k\in \N\), we have
				\begin{equation}
					\label{deltakformula}
					\Delta^k (f) (\xi) (\gamma) = \sum\limits_{\beta \in \mathcal G_{s(\gamma)}} \left (\ell (\gamma) - \ell (\beta) \right )^k f \left (\gamma \beta^{-1} \right ) \xi (\beta).   
				\end{equation}
				Moreover, $\Delta^{k}(f)$ extends as a bounded operator on the Hilbert module $L^{2}(\mathcal G)$. 
			\end{lem}
			\begin{proof}
				We will prove the formula~\eqref{deltakformula} by induction on $k.$ The case $k=1$ has already been proved. Hence, we assume the result to be true for $k - 1.$ Then we have
				\Bea
				\Delta^k (f) (\xi) (\gamma) & = & \big [D, \Delta^{k-1} (f) \big ] (\xi) (\gamma) \\ & = & M_{\ell} \left (\Delta^{k - 1} (f) (\xi) \right ) (\gamma) - \Delta^{k -1} (f) \left (M_{\ell} (\xi) \right ) (\gamma) \\ & = & \ell (\gamma) \Delta^{k - 1} (f) (\xi) (\gamma) - \Delta^{k -1} (f) \left (M_{\ell} (\xi) \right ) (\gamma) \\ & = & \ell (\gamma) \sum\limits_{\beta \in \mathcal G_{s(\gamma)}} \left (\ell (\gamma) - \ell (\beta) \right )^{k -1} f \left (\gamma \beta^{-1} \right ) \xi (\beta) \\ &&  - \sum\limits_{\beta \in \mathcal G_{s(\gamma)}} \left (\ell (\gamma) - \ell (\beta) \right )^{k - 1} f \left (\gamma \beta^{-1} \right ) \ell (\beta) \xi (\beta) \\ & = & \sum\limits_{\beta \in \mathcal G_{s(\gamma)}} \left (\ell (\gamma) - \ell (\beta) \right )^k f \left (\gamma \beta^{-1} \right ) \xi (\beta)  
				\Eea
				
				where the penultimate equality follows from the induction hypothesis. Hence, the claim follows.
				
				Thus for $f, \xi \in C_c (\mathcal G),$ we have
				\Bea
					\hnorm{\Delta^k (f) (\xi)}^2  & = & \sup\limits_{x \in \mathcal G^{(0)}} \sum\limits_{\gamma \in \mathcal G_x} \bigl \lvert \Delta^k (f) (\xi) (\gamma) \bigr \rvert^2 \\ & = &  \sup\limits_{x \in \mathcal G^{(0)}} \sum\limits_{\gamma \in \mathcal G_x} \biggl \lvert \sum\limits_{\beta \in \mathcal G_x} \left (\ell (\gamma) - \ell (\beta) \right )^k f \left (\gamma \beta^{-1} \right ) \xi (\beta) \biggr \rvert^2 \\ &
					\leq & \sup\limits_{x \in \mathcal G^{(0)}} \sum\limits_{\gamma \in \mathcal G_x} \biggl ( \sum\limits_{\beta \in \mathcal G_x} \left \lvert \ell (\gamma) - \ell (\beta) \right \rvert^k \left \lvert f \left (\gamma \beta^{-1} \right ) \right \rvert \left \lvert \xi (\beta) \right \rvert \biggr )^2 \\ & \leq &  \sup\limits_{x \in \mathcal G^{(0)}} \sum\limits_{\gamma \in \mathcal G_x} \biggl ( \sum\limits_{\beta \in \mathcal G_x} \ell \left (\gamma \beta^{-1} \right )^k \left \lvert f \left (\gamma \beta^{-1} \right ) \right \rvert \left \lvert \xi (\beta) \right \rvert \biggr )^2 \quad {(\mathrm{by\ subadditivity\ of}\ \ell)} \\
					 & = & \sup\limits_{x \in \mathcal G^{(0)}} \sum\limits_{\gamma \in \mathcal G_x} \biggl ( \bigl (M_{\ell}^k \left (\lvert f \rvert \bigr ) \ast \lvert \xi \rvert \right ) (\gamma) \biggr )^2 \\ & = & \left \|M_{\ell}^{k} \left (\lvert f \rvert  \right ) \ast \lvert \xi \rvert \right \|^2 \\ & = & \left \|\lambda \left (M_{\ell}^k \left (\lvert f \rvert \right )  \right )  \left (\lvert \xi \rvert \right ) \right \|^2 \\ & 
					 \leq & \left \|\lambda \left (M_{\ell}^k \left (\lvert f \rvert \right )  \right ) \right \|_{\mathrm {adj}}^2 \|\xi\|^2  \\ & \leq &  \left \|M_{\ell}^k \left (\lvert f \rvert \right ) \right \|_{I}^2 \|\xi\|^2  \\ & = &  \left \|M_{\ell}^k (f) \right \|_{I}^2 \|\xi\|^2.
				\Eea
		 Therefore, we have the inequality \(\left \|\Delta^k (f) \right \| \leq \left \|M_{\ell}^k (f) \right \|_{I}\) for any function \(f\) in \(C_c (\mathcal G)\).
	 This allows us to extend $\Delta^k (f)$ as a bounded operator on $L^2 (\mathcal G)$ as follows:
				$$\Delta^k (f) (\xi) = \lim\limits_{n \to \infty} \Delta^k (f) (\xi_n)$$
				where $\{\xi_n \}_{n \geq 1}$ is any sequence in $C_c (\mathcal G)$ converging to $\xi$ in the Hilbert module norm $\|\cdot\|.$
			\end{proof}
			\begin{prop}\label{prop21}
				For $f \in C_c (\mathcal G),$ the operator $\Delta^k (f)$ is adjointable and 
				$$\Delta^k (f)^{\ast} = (-1)^k \Delta^k (f^{\ast}).$$ 
			\end{prop}
			
			\begin{proof}
				Since $C_c (\mathcal G)$ is dense in $L^2 (\mathcal G)$ and $\Delta^k (f)$ is bounded, in order to show the desired equality it is enough to show that for any $\xi, \eta \in C_c (\mathcal G),$ we have 
				
				$$\bigl \langle \bigl \langle \Delta^k (f) (\xi), \eta \bigr \rangle \bigr \rangle = (-1)^k \bigl \langle \bigl \langle \xi, \Delta^k (f^{\ast}) (\eta) \bigr \rangle \bigr \rangle.$$
				
				We prove the above equality for the case $k=1$. For the general case a regular induction argument will suffice and we leave the induction argument to the reader. To that end, for $f, \xi, \eta \in C_c (\mathcal G)$ and $x \in \mathcal G^{(0)}$ we have
				\Bea
				\left \langle \left \langle M_{\ell} (\lambda (f) (\xi)), \eta \right \rangle \right \rangle (x)
				= \sum\limits_{\gamma \in \mathcal G_x} \ell (\gamma) \overline {\lambda (f) (\xi) (\gamma)}\ \eta (\gamma)  & = &  \sum\limits_{\gamma \in \mathcal G_x} \ell (\gamma) \sum\limits_{\beta \in \mathcal G_x} \overline {f \left (\gamma \beta^{-1} \right )}\ \overline {\xi (\beta)}\ \eta (\gamma) \\  & = &  \sum\limits_{\gamma \in \mathcal G_x} \sum\limits_{\beta \in G_x} \ell (\gamma) \overline {f \left (\gamma \beta^{-1} \right )}\ \overline {\xi (\beta)}\ \eta (\gamma).
				\Eea
				Using Fubini's theorem the last quantity can be written as
				\Bea
				\sum\limits_{\beta \in \mathcal G_x} \sum\limits_{\gamma \in \mathcal G_x} \ell (\gamma) \overline {f \left (\gamma \beta^{-1} \right )}\ \overline {\xi (\beta)}\ \eta (\gamma)  & = &  \sum\limits_{\beta \in \mathcal G_x} \overline {\xi (\beta)}\ \sum\limits_{\gamma \in \mathcal G_x} \ell (\gamma) f^{\ast} \left (\beta \gamma^{-1} \right )\ \eta (\gamma) \\  & = &  \sum\limits_{\beta \in \mathcal G_x} \overline {\xi (\beta)}\ \sum\limits_{\gamma \in \mathcal G_x} f^{\ast} \left (\beta \gamma^{-1} \right )\ M_{\ell} (\eta) (\gamma) \\ & = &  \sum\limits_{\beta \in \mathcal G_x} \overline {\xi (\beta)} \left (f^{\ast} \ast M_{\ell} (\eta) \right ) (\beta) \\ & = & \sum\limits_{\beta \in \mathcal G_x} \overline {\xi (\beta)} \lambda (f^{\ast}) \left (M_{\ell} (\eta) \right ) (\beta) \\ & = & \left \langle \left \langle \xi, \lambda (f^{\ast}) \left (M_{\ell} (\eta) \right ) \right \rangle \right \rangle (x).  
				\Eea
				Again for \(f, \xi, \eta \in C_c (\mathcal G)\) and \(x\in \mathcal{G}^{(0)}\), we have
				\Bea
					\left \langle \left \langle \lambda (f) \left (M_{\ell} (\xi) \right ), \eta \right \rangle \right \rangle (x) =\sum\limits_{\gamma \in \mathcal G_x} \overline {\lambda (f) \left (M_{\ell} (\xi) \right ) (\gamma)}\ \eta (\gamma) & = & \sum\limits_{\gamma \in \mathcal G_x} \overline {\left (f \ast M_{\ell} (\xi) \right ) (\gamma)}\ \eta (\gamma) \\ & = &  \sum\limits_{\gamma \in \mathcal G_x} \sum\limits_{\beta \in \mathcal G_x} \overline {f \left (\gamma \beta^{-1} \right )}\ \overline {M_{\ell} (\xi) (\beta)}\ \eta (\gamma).
				\Eea
				Again, using Fubini's theorem the last term becomes
				\Bea
				\sum\limits_{\beta \in \mathcal G_x} \sum\limits_{\gamma \in \mathcal G_x} \overline {f \left (\gamma \beta^{-1} \right )}\ \ell (\beta)\ \overline {\xi (\beta)}\ \eta (\gamma)  & = &  \sum\limits_{\beta \in \mathcal G_x} \ell (\beta) \overline {\xi (\beta)} \sum\limits_{\gamma \in \mathcal G_x} \overline {f \left (\gamma \beta^{-1} \right )}\ \eta (\gamma) \\ & = &  \sum\limits_{\beta \in \mathcal G_x} \ell (\beta) \overline {\xi (\beta)} \sum\limits_{\gamma \in \mathcal G_x} f^{\ast} \left (\beta \gamma^{-1} \right )\ \eta (\gamma) \\ & = &  \sum\limits_{\beta \in \mathcal G_x} \ell (\beta) \overline {\xi (\beta)} \left (f^{\ast} \ast \eta \right ) (\beta) \\ & = &  \sum\limits_{\beta \in \mathcal G_x} \overline {\xi (\beta)} M_{\ell} \left (f^{\ast} \ast \eta \right ) (\beta) \\ & = & \left \langle \left \langle \xi, M_{\ell} \left (\lambda (f^{\ast}) (\eta) \right ) \right \rangle \right \rangle (x).
				\Eea

Therefore, we have 
\Bea
	\left \langle \left \langle \Delta (f) (\xi), \eta \right \rangle \right \rangle & = &  \left \langle \left \langle M_{\ell} (\lambda (f) (\xi)), \eta \right \rangle \right \rangle - \left \langle \left \langle \lambda (f) (M_{\ell} (\xi)), \eta \right \rangle \right \rangle \\ & = &  \left \langle \left \langle \xi, \lambda (f^{\ast}) \left (M_{\ell} (\eta) \right ) \right \rangle \right \rangle - \left \langle \left \langle \xi, M_{\ell} \left (\lambda (f^{\ast}) (\eta) \right ) \right \rangle \right \rangle \\ & = &  - \left \langle \left \langle \xi, \Delta (f^{\ast}) (\eta) \right \rangle \right \rangle.  
\Eea
				
		This shows that $\Delta (f)^{\ast} = - \Delta (f^{\ast})$ for any $f \in C_c (\mathcal G),$  proving the equality for $k = 1.$ 
			\end{proof}
			
			
			
			\begin{lem}\label{lem2}
				Let $\mathcal G$ be an étale groupoid. Then for any $f, \xi \in C_c (\mathcal G)$ and $k \geq 1,$ $\Delta^k (f) (\xi) \in C_c (\mathcal G).$    
			\end{lem}
			
			\begin{proof}
				Clearly, the result holds for $k = 1.$ We assume that the result holds for $k - 1$ for $k > 1.$ Since $M_{\ell} (\eta) \in C_c (\mathcal G)$ for all $\eta \in C_c (\mathcal G),$ by induction hypothesis it turns out that both $M_{\ell} \left (\Delta^{k - 1} (\xi) \right )$ and $\Delta^{k - 1} \left (M_{\ell} (\xi) \right )$ are in $C_c (\mathcal G)$ and consequently, $\Delta^k (f) (\xi) \in C_c (\mathcal G).$
			\end{proof}
			
			\begin{prop}\label{prop22}
				Let $\mathcal G$ be an étale groupoid such that the unit space $\mathcal G^{(0)}$ is compact. Let $f \in C_c (\mathcal G),$ $\ell$ be a length function on $\mathcal G$ and let $\Delta^k (f)$ be the corresponding adjointable operator on $L^2 (\mathcal G).$ Then $\Delta^k (f) = 0$ if and only if $f (\gamma) = 0$ for all $\gamma \in \mathcal G \setminus \mathcal G^{(0)},$ that is, $f \in C \left (\mathcal G^{(0)} \right ).$    
			\end{prop}
			
			\begin{proof}
				Let $\Delta^k (f) = 0$ for some $f \in C_c (\mathcal G).$ Then $\Delta^k (f) (\mathcal E) = 0$ in $L^2 (\mathcal G),$ where $\mathcal E$ is the unit given by $$\mathcal E (\gamma) = \begin{cases} 1, \quad \text {if}\ \gamma \in \mathcal G^{(0)}, \\ 0, \quad \text {otherwise}. \end{cases}$$ 
				Since $\mathcal G^{(0)}$ is compact and the groupoid \(\mathcal{G}\) is {\'e}tale, it follows that $\mathcal E \in C_c (\mathcal G)$. Using Lemma~\ref{lem2} we conclude that $\Delta^k (f) (\mathcal E) \equiv 0,$ as $C_c (\mathcal G)$ is faithfully embedded in $L^2 (\mathcal G).$ Now for any $\gamma \in \mathcal G$ 
				\[
				\Delta^k (f) (\mathcal E) (\gamma)  =  \sum\limits_{\beta \in \mathcal G_{s(\gamma)}} \left (\ell (\gamma) - \ell (\beta) \right )^k f \left (\gamma \beta^{-1} \right ) \mathcal E (\beta)  =  \ell (\gamma)^k f (\gamma).
				\]
			Therefore, $\Delta^k (f) (\mathcal E) \equiv 0$ if and only if $\ell (\gamma)^k f(\gamma) = 0$ for all $\gamma \in \mathcal G.$ Since $\ell (\gamma) = 0$ if and only if $\gamma \in \mathcal G^{(0)},$ it follows that $f(\gamma) = 0$ for all $\gamma \in \mathcal G \setminus \mathcal G^{(0)},$ that is, $f \in C \left (\mathcal G^{(0)} \right ).$

				Conversely, assume that $f \in C \left (\mathcal G^{(0)} \right ).$ Then for any $\xi \in C_c (\mathcal G)$ and $\gamma \in \mathcal G$, we have
				\[
					\Delta^k (f) (\xi) (\gamma)  =  \sum\limits_{\beta \in \mathcal G_{s(\gamma)}} \left (\ell (\gamma) - \ell (\beta) \right )^k f \left (\gamma \beta^{-1} \right ) \xi (\beta)  =  \left (\ell (\gamma) - \ell (\gamma) \right )^k f \left (\gamma \gamma^{-1} \right ) \xi (\gamma) =  0.
				\] 
			Since $C_c (\mathcal G)$ is faithfully embedded inside $L^2 (\mathcal G),$ it follows that $\Delta^k (f) (\xi) = 0$ in $L^2 (\mathcal G),$ for any $\xi \in C_c (\mathcal G)$ and consequently, $\Delta^k (f) = 0.$ This completes the proof.
			\end{proof}
			
			Now we will discuss compact quantum metric spaces as defined in~\cite{Connes-1989-Compact-metric-sp-Fred-mod-Hyperfinteness,Article}. We wish to define a Lip-norm using the sequence of adjointable operators $\Delta^{k}$ on $C^{\ast}_{r}(\mathcal G)$ as done in~\cite{Long-Wu-2017-Twisted-group-C-alg-as-CQMS}.
			 However, as observed in the previous proposition, the kernel of \(\Delta^{k}\) is quite large for all \(k \in \mathbb{N}\). Given this situation, we will define the following (see also~\cite[Section 3]{Long-Wu-2017-Twisted-group-C-alg-as-CQMS}).
			
			\begin{defn}\label{def6} We call a seminorm $L$ on a unital $C^{\ast}$-algebra $A$ a Lipschitz {\it quasi}-seminorm if the following conditions are satisfied:
				
				$(1)$ $L(a^{\ast}) = L(a)$ for all $a \in A$;
				
				$(2)$ the set $\mathcal K : = \{a \in A\ :\ L(a) = 0 \}$ is a closed subspace of $A$ such that $1_{A}\in\mathcal{K}$;
				
				$(3)$ the set $\mathcal A : = \{a \in A\ :\ L(a) < +\infty \}$ is a dense subset of $A.$
				
			\end{defn}
			Note that a quasi-seminorm is a seminorm in the sense of~\cite[Section 2]{Rieffel-1999-Metrics-on-state-spaces} if ${\rm dim}(\mathcal K)=1$. Given a Lipschitz quasi-seminorm $L$ on a $C^{\ast}$-algebra as above one can define a metric $\rho$ on the state space $\mathcal S(A)$ by the following formula (see~\cite[Equation 1.1]{Rieffel-1998-Metrics-on-state-from-action-of-cmpt-gp}):
			\begin{displaymath}
				\rho(\mu,\nu)={\rm sup}\{|\mu(a)-\nu(a)|:a\in A, \ L(a)\leq 1\}.
			\end{displaymath}
			However, if ${\rm dim}(\mathcal K)\geq 2$, the metric $\rho$ always attains the value $+\infty$ and therefore it cannot metrize any connected topology (see \cite{https://doi.org/10.1112/blms.12930}) and in particular the weak$^\ast$-topology on the state space $\mathcal S(A)$. 
			Let $A$ be a $C^{\ast}$-algebra and $L$ be a quasi-seminorm on $A.$ Let $\eta$ be a state on $\mathcal K : = \{a \in A\ :\ L(a) = 0 \}.$ Let $A^{\ast}$ be the Banach space dual of $A,$ and set 
			$$S_{\eta} : = \left \{\mu \in A^{\ast}\ :\ \mu \rvert_{\mathcal K} = \eta,\ \text {and}\ \|\mu\| = 1 \right \}.$$
			
			It is clear that \( S_{\eta} \) is a weak\(^*\)-closed subset of the closed unit ball of \( A^{\ast} \) and, therefore, by Banach-Alaoglu's theorem, it is weak\(^*\)-compact. The following lemma is well-known. However, we could not find an explicit reference for it. So we briefly mention it here.
			\begin{lem}
				Let ${\rm dim}(\mathcal K)\geq 2$; $\phi\in S_\eta, \psi\in S_{\eta^{\prime}}$ for $\eta\neq\eta^{\prime}$. Then $\rho(\phi,\psi)=+\infty$. 
			\end{lem}
			\begin{proof}
				As $\eta\neq\eta^{\prime}$, there is some $a\in\mathcal K$ such that $\eta(a)\neq\eta^{\prime}(a)$. As $\mathcal K$ is a subspace, we can assume without loss of generality that $|\eta(a)-\eta^{\prime}(a)|=1$. Then as $\phi\in S_{\eta}, \psi\in S_{\eta^{\prime}}$, we have 
				\begin{displaymath}
					|\phi(a)-\psi(a)|=|\eta(a)-\eta^{\prime}(a)|=1.
				\end{displaymath}
				We fix an arbitrary $N\in\mathbb{N}$. Then $Na\in\mathcal K$ and therefore $\rho(\phi,\psi)\geq N$. As $N$ is arbitrary, this proves the lemma.
			\end{proof}
			
			
			
			
			The above lemma indicates that for ${\rm dim}(\mathcal K)\geq 2$ if we want a genuine metric induced from a quasi-seminorm then we have to look at the `restriction' of the metric $\rho$ on the weak$^\ast$-compact spaces $S_{\eta}$. We denote the restriction by $\rho_{L}$.
			We will refer to the topology on $S_{\eta}$ defined by $\rho_L$ as the ``$\rho_L$-topology,'' or the ``metric-topology'' when $\rho_L$ is understood.
			
			\begin{defn}\label{def10}
				We call a Lipschitz quasi-seminorm $L$ a quasi-Lip-norm on a unital $C^{\ast}$-algebra $A$ if the metric topology on $S_{\eta}$ induced from $\rho_L$ coincides with the weak$^{\ast}$-topology for all states $\eta$ on $\mathcal K.$ If there is a quasi-Lip-norm $L$ on $A,$ we say that that the pair $(A, L)$ is a quasi-compact quantum metric space. 
			\end{defn}
			Now we come back to étale groupoids. From now on we work with étale groupoids such that the unit space $\mathcal G^{(0)}$ is compact. Let $L$ be a Lipschitz quasi-seminorm on $C^{\ast}_{r}(\mathcal G)$ such that $C(\mathcal G^{(0)})\subseteq {\rm ker}(L)$ where $\mathcal G$ is an étale groupoid and $\mathcal G^{(0)}$ is compact. First we prove a sufficient condition so that $L$ becomes a quasi-Lip-norm according to the Definition~\ref{def10}. The following lemma is an analogue of~\cite[Theorem 1.8]{Rieffel-1998-Metrics-on-state-from-action-of-cmpt-gp}.
			\begin{lem}\label{lem20}
				Let $\mathcal G$ be an étale groupoid such that the unit space $\mathcal G^{(0)}$ is compact. Let $\eta$ be a state on $C(\mathcal G^{(0)}).$ Let $L$ be a Lipschitz quasi-seminorm on $C_r^{\ast} (\mathcal G)$ such that $C(\mathcal G^{(0)})\subseteq {\mathrm {ker}}(L)$. If the set $\mathcal L_1' : = \left \{f \in C_c (\mathcal G)\ :\ L(f) \leq 1,\ f\rvert_{\mathcal G^{(0)}} \equiv 0 \right \}$ is totally bounded in $C_r^{\ast} (\mathcal G)$ with respect to the reduced norm $\|\cdot\|_{\mathrm {red}},$ then the weak*-topology on $S_{\eta}$ coincides with the $\rho_{L}$-topology for any state $\eta$ on $\mathrm {ker} (L)$. This means that the pair \((C_r^{\ast}(\mathcal{G}), L)\) forms a quasi-compact quantum metric space.    
			\end{lem}
			
			\begin{proof}
				Let $\mu \in S_{\eta}$ and $\varepsilon > 0$ be given and let $B (\mu, \varepsilon)$ be the $\rho_{L}$ ball of radius $\varepsilon$ about $\mu$ in~$S_{\eta}.$ It suffices to show (see \cite[Proposition 1.4 and Theorem 1.8]{Rieffel-1998-Metrics-on-state-from-action-of-cmpt-gp}) that $B(\mu, \varepsilon)$ contains a weak*-neighbourhood of $\mu.$ Using the total boundedness of $\mathcal L_1'$ we choose $f_1, \cdots, f_n \in \mathcal L_1'$ such that $\|\cdot\|_{\mathrm {red}}$-balls of radius $\frac {\varepsilon} {4}$ about $f_j$'s cover $\mathcal L_1'.$ We will show that the weak*-neighbourhood 
				\[
				 \mathcal O  =  \mathcal O \left (\mu, \{f_j\}, \frac {\varepsilon} {4} \right )  : =  \left \{\nu \in S_{\eta}\ :\ \left \lvert (\mu - \nu) (f_j) \right \rvert < \frac {\varepsilon} {4},\ 1 \leq j \leq n \right \}
				\]
			 is contained in $B(\mu, \varepsilon).$ Take any $f \in C_c (\mathcal G)$ with $L (f) \leq 1.$ Then $f' : = f - f \rvert_{\mathcal G^{(0)}} \in \mathcal L_1'.$ Choose \(f_i\) such that $\|f' - f_i\|_{\mathrm {red}} < \frac {\varepsilon} {4}$ for \(1\leq i\leq n\).  Then for any $\nu \in \mathcal O$ we have 
			 \Bea
			 		\left \lvert \mu (f) - \nu (f) \right \rvert = \left \lvert \mu (f') - \nu (f') \right \rvert  & \leq &  \left \lvert \mu (f' - f_i) - \nu (f' - f_i) \right \rvert + \left \lvert \mu (f_i) - \nu (f_i) \right \rvert \\  & \leq &  \|\mu - \nu\| \|f'- f_i \|_{\mathrm {red}} + \left \lvert (\mu - \nu) (f_i) \right \rvert \\ & < &  \frac {2 \varepsilon} {4} + \frac {\varepsilon} {4} =  \frac {3\varepsilon} {4}.   
			 	\Eea
				This shows that $\rho_{L} (\mu, \nu) \leq \frac {3 \varepsilon} {4} < \varepsilon.$ Hence $\nu \in B(\mu, \varepsilon),$ i.e., $\mathcal O \subseteq B(\mu, \varepsilon),$ as required.
			\end{proof}

			Now we define a sequence of quasi-seminorms on $C^{\ast}_{r}(\mathcal G)$ for an étale groupoid $\mathcal G.$ The definition is inspired by \cite[Section 3]{Long-Wu-2017-Twisted-group-C-alg-as-CQMS}. To that end, given a length function $\ell$ on $\mathcal G$, recall the sequence of adjointable operators $\Delta^{k}$ from Proposition~\ref{prop21}.
			\begin{defn}\label{defn12}
				Let $\ell :\mathcal{G}\to [0,+\infty)$ be a length function on an etale groupoid $\mathcal{G}$.  For any $k \in \mathbb N,$ we define $L_{\ell}^{k}:C_{r}^{\ast}(\mathcal{G})\to [0,+\infty]$ by
				
				$$L_{\ell}^k (a) : = \begin{cases} \left \|\Delta^k (a) \right \|_{\mathrm {adj}}, \quad a \in C_c (\mathcal G), \\ + \infty, \quad a \in C_r^{\ast} (\mathcal G) \setminus C_c (\mathcal G). \end{cases}$$
			\end{defn}
			
			\begin{lem}
				Let $\mathcal G$ be an étale groupoid equipped with a length function $\ell$. Then for any natural number $k,$ $L_{\ell}^k$ is a Lipschitz quasi-seminorm on $C_r^{\ast} (\mathcal G).$ The kernel of $L_{\ell}^{k}$ is $C(\mathcal G^{(0)})$ for all $k$.
			\end{lem}
			
			\begin{proof}
				It is clear that $L_{\ell}^k$ is a seminorm on $C_r^{\ast} (\mathcal G)$ permitted to take the value $+\infty.$ The result then follows directly from Proposition~\ref{prop21} and Proposition~\ref{prop22}.
			\end{proof}
			
			Let \(\mathcal{G}\) be an {\'e}tale groupoid with a length function \(\ell\). For \(f\in C_c(\mathcal{G})\), we define
			\[
			 	\|f\|_{2,p,s,\ell} := \sup\limits_{x \in \mathcal G^{(0)}} \biggl (\sum\limits_{\gamma \in \mathcal G_x} \left \lvert f(\gamma) \right \rvert^2 (1 + \ell (\gamma))^{2 p} \biggr )^{\frac {1} {2}} \textup{ and } \|f\|_{2,p,r, \ell} : = \sup\limits_{x \in \mathcal G^{(0)}} \biggl (\sum\limits_{\gamma \in \mathcal G_x} \left \lvert f \left (\gamma^{-1} \right ) \right \rvert^2 (1 + \ell (\gamma))^{2 p} \biggr )^{\frac {1} {2}}
			\]
			where \(p>0\). 
			\begin{defn}
				Let $\mathcal G$ be an étale groupoid. We say that $\mathcal G$ has the rapid decay property with respect to the length function \(\ell\) if there exist constants $C, p > 0$ such that
				$\|f\|_{\mathrm {red}} \leq C \|f\|_{2, p, \ell},$
				for all $f \in C_c (\mathcal G),$ where
				\[
				\hnorm{f}_{2,p,\ell} = \max \{ \hnorm{f}_{2,p,s,\ell}, \hnorm{f}_{2,p,r,\ell}\}.
				\]
			\end{defn}

			For the rest of the paper, we focus on the case where \(\mathcal{G}\) is the transformation groupoid \(\Gamma \ltimes X\), with a compact Hausdorff space \(X\), as discussed in Example~\ref{ex2}. The following observation will be used to identify the decomposition of the state space \(\mathcal{S}(C^{\ast}_{r}(\mathcal{G}))\).
			 
			 \begin{observation}\label{lem3}
			 	Let $\Gamma$ be a discrete group acting on a compact Hausdorff space $X.$ Consider the associated transformation groupoid $\mathcal G : = \Gamma \ltimes X.$  Then $C(X)$ with pointwise multiplication and the sup norm is a $C^{\ast}$-subalgebra of $C_r^{\ast} (\mathcal G).$
			 \end{observation}

			From now on, we will be working with the sequence of Lipschitz quasi-seminorms mentioned in Definition~\ref{defn12}.
			Based on the previous observation, it can be established that for any \( k \in \mathbb{N}\), the state space of \( \text{ker}(L_{\ell}^{k}) (= C(X)) \) can be identified with the space of probability measures on \( X \) via the Riesz representation theorem. We will continue to denote the probability measure corresponding to a state \( \eta\) on \(C(X)\) simply as \(\eta\). Let us denote the set of probability measures on \( X \) by \( \mathcal{M}(X) \). We thus have the following decomposition:
			
		\begin{displaymath}
				\mathcal{S}(C^{\ast}_{r}(\mathcal G))=\bigcup\limits_{\eta\in\mathcal{M}(X)}S_{\eta},
			\end{displaymath}
			where $\mathcal G=\Gamma\ltimes X$ for some compact Hausdorff space $X$.
			
			\begin{lem}\label{lem4}
				Let $\Gamma$ be a discrete group acting on a compact Hausdorff space $X.$ Let $\mathcal G : = \Gamma \ltimes X$ be the associated transformation groupoid equipped with a continuous proper length function $\ell$; let $p > 0.$ Then there exists a real number $\alpha > 0$ such that
				$$\|f\|_{2, p, \ell} \leq \alpha L_{\ell}^k (f)$$ \vspace{2mm}
				for any  $f \in C_c (\mathcal G)$ with $f \rvert_{X} \equiv 0$ and any integer $k > p.$
			\end{lem}
			
			\begin{proof}
				Recall the multiplicative identity $\mathcal E$ in $C_r^{\ast} (\mathcal G)$ given by   
				
				$$\mathcal E (g, x) = \begin{cases} 1, \quad \text {if}\ g = e, \\ 0, \quad \text {otherwise}. \end{cases}$$ 
				Let $f \in C_c (\mathcal G)$.  Since 
			   $\|\Delta^k (f)\|_{\mathrm {adj}} = \|\Delta^k (f^{\ast})\|_{\mathrm {adj}}$ and $\|\mathcal E\| = 1$, it turns out that 
			   \begin{equation}\label{equ-Lip-norm-ineq}
			   	\max \left (\left \|\Delta^k (f) (\mathcal E) \right \|^2, \left \|\Delta^k (f^{\ast}) (\mathcal E) \right \|^2 \right ) \leq \left (L_{\ell}^k (f) \right )^2.
			   \end{equation}
			
				Now, 
				
				\Bea
				\left \|\Delta^k (f) (\mathcal E) \right \|^2 & = & \sup\limits_{x \in X} \sum\limits_{g \in \Gamma} \left \lvert \Delta^k (f) (\mathcal E) (g, x) \right \rvert^2 \\ & = & \sup\limits_{x \in X} \sum\limits_{g \in \Gamma} \left \lvert \sum\limits_{h \in \Gamma} \left (\ell (g, x) - \ell (h, x) \right )^k f \left (gh^{-1}, h \cdot x \right )\ \mathcal {E} (h, x) \right \rvert^2 \\ & = & \sup\limits_{x \in X} \sum\limits_{g \in \Gamma} \ell (g, x)^{2k} \left \lvert f(g, x) \right \rvert^2.
				\Eea
				
				Similarly, by replacing $f$ by $f^{\ast}$, we have
				\[
				\left \|\Delta^k (f^{\ast}) (\mathcal E) \right \|^2  = \sup\limits_{x \in X} \sum\limits_{g \in \Gamma} \ell (g, x)^{2k} \left \lvert f^{\ast} (g, x) \right \rvert^2  =  \sup\limits_{x \in X} \sum\limits_{g \in \Gamma} \ell (g, x)^{2k} \left \lvert f \left (g^{-1}, g \cdot x \right ) \right \rvert^2. 
				\]
	        Using the above computation in~\eqref{equ-Lip-norm-ineq}, we have
				
				$$\max \biggl (\sup\limits_{x \in X} \sum\limits_{g \in \Gamma} \ell (g, x)^{2k} \left \lvert f(g, x) \right \rvert^2, \sup\limits_{x \in X} \sum\limits_{g \in \Gamma} \ell (g, x)^{2k} \left \lvert f \left (g^{-1}, g \cdot x \right ) \right \rvert^2 \biggr ) \leq \bigl (L_{\ell}^k (f) \bigr )^2.$$
			For any $g \in \Gamma$ with $\ell (g, x) \geq n \geq 1$ and $k > p$ we have 
				
				$$\left (1 + \ell (g, x) \right )^{2p} \leq 2^{2p} \ell (g, x)^{2p} \leq 2^{2p} n^{2p - 2k} \left (\ell (g, x) \right )^{2k}.$$
				
				Since $\ell$ is a proper length function, for any $n \in \mathbb N,$ there exists a finite subset $\Gamma_n \subseteq \Gamma$ such that the set 
				$$\left \{(g, x) \in \mathcal G\ :\ \ell (g, x) \leq n \right \} \subseteq \Gamma_n \times X.$$
				Then by similar estimation as in \cite[Lemma 3.3]{Long-Wu-2017-Twisted-group-C-alg-as-CQMS}, we have

				$$\sum\limits_{g \in \Gamma} \left \lvert f(g, x) \right \rvert^2 \left (1 + \ell (g, x) \right )^{2p} \leq \alpha^2 \left (L_{\ell}^k (f) \right )^2$$
				where 
				
				$$\alpha : = \sup\limits_{\substack {g \in \Gamma_n \setminus \{e\} \\ x \in X}} \left (\left (1 + \ell (g, x)^{-1} \right )^{2k} + 2^{2p} n^{2p - 2k} \right )^{\frac {1} {2}} < + \infty.$$
				Since both $\alpha$ and $L_{\ell}^k (f)$ are independent of $x \in X,$ it follows that
				
				$$\sup\limits_{x \in X} \biggl(\sum\limits_{g \in \Gamma} \left \lvert f(g, x) \right \rvert^2 \left (1 + \ell (g, x) \right )^{2p}\biggr)^{\frac{1}{2}} \leq \alpha L_{\ell}^k (f).$$
				Similarly, we can show that
				$$\sup\limits_{x \in X} \biggl(\sum\limits_{g \in \Gamma} \left \lvert f \left (g^{-1}, g \cdot x \right ) \right \rvert^2 \left (1 + \ell (g, x) \right )^{2p} \biggr)^{\frac{1}{2}}\leq \alpha L_{\ell}^k (f).$$
				This completes the proof.
				\end{proof}
			
			\begin{prop}
				Let $\Gamma$ be a discrete group acting on a compact Hausdorff space $X.$ Let $\mathcal G : = \Gamma \ltimes X$ be the associated transformation groupoid equipped with a continuous proper length function $\ell$ and let $p > 0.$ If $\mathcal G$ has the rapid decay property with respect to $\ell$ and $p,$ then for any integer $k > p$ the diameter of the metric space $\big (S_{\eta}, \rho_{L_{\ell}^k} \big )$ is finite, where $\eta$ is a fixed probability measure on $X$. Moreover, the diameter is uniformly bounded over the probability measures on $X$.    
			\end{prop}
			
			\begin{proof}
				Since $\mathcal G$ has the rapid decay property with respect to $\ell$ and $p,$ there is a constant $C > 0$ such that
				$$\|f\|_{\mathrm {red}} \leq C \|f\|_{2, p, \ell},$$ \vspace{4mm}
				for any $f \in C_c (\mathcal G).$ By Lemma \ref{lem4} there is a constant $\alpha > 0$ such that 
				$$\|f\|_{2,p,\ell} \leq \alpha L_{\ell}^k (f)$$
				for any $f \in C_c (\mathcal G)$ with $f \rvert_{X} \equiv 0.$ Let $f \rvert_{X}$ be the extension of the restriction of $f$ to $X$ by zero. Since $X$ is compact and clopen in $\mathcal G,$ $f \rvert_{X} \in C_c (\mathcal G)$ and $\left (f - f \rvert_{X} \right ) \big\rvert_{X} \equiv 0.$ So for any $f \in C_c (\mathcal G)$ we have
				$$\left \|f - f \rvert_{X} \right \|_{\mathrm {red}} \leq C \alpha L_{\ell}^k \left (f - f \rvert_{X} \right ) = C \alpha L_{\ell}^k (f).$$
				Thus
				\vspace{2mm}
				$$\left \|\tilde f \right \|^{\sim} \leq C \alpha L_{\ell}^k (f),$$
				for any $f \in C_c (\mathcal G),$ where $\|\cdot\|^{\sim}$ is the quotient norm on the quotient space $C_c (\mathcal G)/ C(X)$ with respect to the reduced norm $\|\cdot\|_{\mathrm {red}}$ on $C_c (\mathcal G).$ Hence the result follows \cite[Proposition 2.2]{Rieffel-1999-Metrics-on-state-spaces}.
			\end{proof}

			\begin{prop}\label{prop20}
				Let $\Gamma$ be a discrete group acting on a compact Hausdorff space $X.$ Let $\mathcal G : = \Gamma \ltimes X$ be the associated transformation groupoid equipped with a continuous proper length function $\ell.$ Suppose that $\mathcal G$ has the property of rapid decay with respect to the length function $\ell.$ Then the set $\mathcal L_1' : = \left \{f \in C_c (\mathcal G)\ :\ L_{\ell}^k(f) \leq 1,\ f\rvert_X \equiv 0 \right \}$ is totally bounded in $C_r^{\ast} (\mathcal G)$ with respect to the reduced norm $\|\cdot\|_{\mathrm {red}}$ for some $k \geq 1$ if and only if $X$ is finite.
			\end{prop}
			
			\begin{proof}
				To prove the forward implication, we assume on the contrary, that $X$ is infinite. Since $\|f\|_{\infty} \leq \|f\|_{\mathrm {red}},$ for any $f \in C_c (\mathcal G)$ (see~\cite[Lemma 5.6.12.]{brownc}), in order to show that $\mathcal L_1'$ is not totally bounded in $C_r^{\ast} (\mathcal G)$ with respect to the reduced norm $\|\cdot\|_{\mathrm {red}},$ it suffices to show that $\mathcal L_1'$ is not totally bounded with respect to the sup norm $\|\cdot\|_{\infty}$ as a subset of $C_c (\mathcal G).$
				
				\vspace{2mm}
				
				Let $\Gamma'$ be a finite subset of $\Gamma$ not containing the identity $e \in \Gamma.$ Then $\Gamma' \times X$ is a compact subset of $\mathcal G.$ Equip $C (\Gamma' \times X)$ with the sup norm $\|\cdot\|_{\infty}.$ Since $X$ is infinite, $C(\Gamma' \times X)$ is infinite dimensional. Hence there exists a sequence $\{h_m\}_{m \geq 1}$ in $C (\Gamma' \times X)$ with $\|h_m\|_{\infty} = 1$ for all $m \in \mathbb N$ such that $$\|h_p - h_q\|_{\infty} \geq \frac {1} {2},$$ for all $p, q \in \mathbb N.$ For each $m \in \mathbb N,$ extend $h_m$ as a function $\widetilde {h_m}$ on $\mathcal G$ by zero. Since $X$ is clopen and compact, it follows that $\widetilde {h_m} \in C_c (\mathcal G)$ for all $m \in \mathbb N.$ Clearly, $\left \|\widetilde {h_m} \right \|_{\infty} = \|h_m\|_{\infty} = 1$ for all $m \in \mathbb N$ and $$\left \|\widetilde {h_p} - \widetilde {h_q} \right \|_{\infty} = \|h_p - h_q\|_{\infty} \geq \frac {1} {2},$$ for all $p, q \in \mathbb N.$
				
				\vspace{2mm}
				
				Now note that since $\ell$ is continuous and proper, we have $$d_k : = \max \left \{\left \|\ell \rvert_{\Gamma' \times X} \right \|_{\infty}^k, \left \|\ell \rvert_{\Gamma'^{-1} \times X} \right \|_{\infty}^k \right \} < \infty,$$ where $\Gamma'^{-1} : = \{g^{-1} : g \in \Gamma' \}.$  Therefore, $$L_{\ell}^{k} (f) \leq \left \|M_{\ell}^{k} (f) \right \|_{I} \leq d_k \left \lvert \Gamma' \right \rvert \|f\|_{\infty},$$ for any $f \in C_c (\mathcal G),$ where $\left \lvert \Gamma' \right \rvert$ denotes the cardinality of $\Gamma'.$ In particular, $$L_{\ell}^{k} \left (\widetilde {h_m} \right ) \leq d_k \left \lvert \Gamma' \right \rvert : = d_k',$$ for all $m \in \mathbb N.$ Also since $\Gamma'$ does not contain the identity, it follows that $$\widetilde {h_m} \big \rvert_{X} = 0,$$ for all $m \in \mathbb N.$ Thus $\frac {\widetilde {h_m}} {d_k'} \in \mathcal L_1',$ for all $m \in \mathbb N$ and for all $p, q \in \mathbb N,$ $$\left \|\frac {\widetilde {h_p}} {d_k'} - \frac {\widetilde {h_q}} {d_k'} \right \|_{\infty} = \frac {1} {d_k'} \left \|\widetilde {h_p} - \widetilde {h_q} \right \|_{\infty} \geq \frac {1} {2d_k'}.$$ This shows that $\mathcal L_1'$ is not pre-compact in $\left (C_c (\mathcal G), \|\cdot\|_{\infty} \right )$ and hence not totally bounded as a subset of $C_c (\mathcal G)$ with respect to the sup 
				norm $\|\cdot\|_{\infty},$ as required.

				The proof of the converse direction is essentially an adaptation of the proof of \cite[Theorem 3.5]{Long-Wu-2017-Twisted-group-C-alg-as-CQMS}. To prove the converse let us consider the set  
				$$F_n : = \left \{(g, x)\ :\ \ell (g, x) \leq n \right \},$$ for $n \geq 1.$
				Since $\ell$ is proper, there is a finite subset $\Gamma_n \subseteq \Gamma$ such that $F_n \subseteq \Gamma_n \times X.$ Let $\chi_n$ be the characteristic function of $\Gamma_n \times X.$ Then clearly $\chi_n \in C_c (\mathcal G),$ since $\Gamma$ is discrete and $X$ is clopen. For each $n \geq 1,$ let us consider the set
				$$B^{(n)} : = \left \{f \chi_n\ :\ f \in \mathcal L_1' \right \}.$$ Since $\mathcal G$ has the rapid decay property with respect to $\ell$ and $p > 0,$ there exists $C > 0$ such that $$\|f\|_{\mathrm {red}} \leq C \|f\|_{2,p, \ell}.$$ Since $f \rvert_{X} \equiv 0$ for all $f \in \mathcal L_1',$ invoking Lemma \ref{lem4} we have $$\|f \chi_n\|_{\infty} \leq \|f\|_{\infty} \leq \|f\|_{\mathrm {red}} \leq C \|f\|_{2,p, \ell} \leq C \alpha.$$
				 So $B^{(n)}$ is a bounded set in the finite dimensional normed space $\left (C (\Gamma_n \times X), \| \cdot \|_{\infty} \right ).$ Hence $B^{(n)}$ is totally bounded with respect to the sup norm $\|\cdot\|_{\infty}.$ Since any two norms on a finite dimensional space are equivalent, it follows that $B^{(n)}$ is totally bounded with respect to the reduced norm $\|\cdot\|_{\mathrm {red}}$ as well.

				Therefore, in order to prove that $\mathcal L_1'$ is totally bounded it is sufficient to show that for any $\varepsilon > 0$ there exists a natural number $n_0$ such that for all $n \geq n_0$ and for all $f \in \mathcal L_1'$
				$$\left\|f (1 - \chi_n) \right \| < \varepsilon.$$
				
				Fix some $f \in \mathcal L_1'$ arbitrarily. First assume that $\left \|f \left (1 - \chi_n \right ) \right \|_{2, p, \ell} = \left \|f \left (1 - \chi_n \right ) \right \|_{2, p, s, \ell}.$ Then invoking rapid decay property we can get hold of some $C > 0$ such that
				\Bea
				\left \|f \left (1 - \chi_n \right ) \right \|_{\mathrm {red}}^2 & \leq & C \sup\limits_{x \in X} \sum\limits_{g \in \Gamma} \left \lvert f (g, x) \right \rvert^2 \left (1 - \chi_n \right ) (g, x) (1 + \ell (g, x))^{2p} \\ & = & C \sup\limits_{x \in X} \sum\limits_{g \in \Gamma \setminus \Gamma_n} \left \lvert f (g, x) \right \rvert^2 (1 + \ell (g, x))^{2p}.  
				\Eea
				Also, if $\ell (g, x) \geq n \geq 1,$ then for all $k \geq k_0 : = \lfloor p \rfloor + 1$ we have 
				$$(1 + \ell (g, x))^{2p} \leq 2^{2p} (\ell (g, x))^{2p} \leq 2^{2p} n^{2p - 2k} (\ell(g, x))^{2k}.$$
				Then similar estimation as in \cite[Lemma 3.3]{Long-Wu-2017-Twisted-group-C-alg-as-CQMS} yields
				$$\left \|f \left (1 - \chi_n \right ) \right \|_{\mathrm {red}} < \varepsilon,$$ for sufficiently large $n$ and for any $f \in \mathcal L_1',$ whenever $\left \|f \left (1 - \chi_n \right ) \right \|_{2, p, \ell} = \left \|f \left (1 - \chi_n \right ) \right \|_{2, p, s, \ell}.$ Similarly, whenever $\left \|f \left (1 - \chi_n \right ) \right \|_{2, p, \ell} = \left \|f \left (1 - \chi_n \right ) \right \|_{2, p, r, \ell},$ invoking rapid decay property we can ensure that $\left \|f \left (1 - \chi_n \right ) \right \|_{\mathrm {red}} < \varepsilon$ for sufficiently large $n$, completing the proof.
			\end{proof}
			
			An immediate corollary is the following.
			
			\begin{cor}\label{cor20}
				Let $\Gamma$ be a discrete group acting on a finite set $X.$ Let $\mathcal G : = \Gamma \ltimes X$ be the associated transformation groupoid equipped with a proper length function $\ell.$ If $\mathcal G$ has the rapid decay property with respect to the length function $\ell$ and $p > 0,$ then the metric $\rho_{L_{\ell}^k}$ induces weak$^{\ast}$-topology on $S_{\eta},$ for all $k > p$ and for any probability measure $\eta$ on $X.$ Therefore, $\left (C_r^{\ast} (\mathcal G), L^k_{\ell} \right )$ is a quasi compact quantum metric space for all $k > p.$    
			\end{cor}
			
			\begin{proof}
				The proof immediately follows from Lemma \ref{lem20} and Proposition \ref{prop20}.
			\end{proof}
			\begin{rem}
				Note that the sufficient condition in Proposition \ref{prop20} is not necessary so that when $X$ has infinitely many points in the transformation groupoid $\mathcal G$, we cannot conclude anything from  Proposition \ref{prop20}. The difficulty in proving the necessity of the condition in Lemma \ref{lem20} has already been noted in \cite[Theorem 1.8]{Rieffel-1998-Metrics-on-state-from-action-of-cmpt-gp}. To prove that the condition necessary one needs an inequality like \cite[Condition 1.5]{Rieffel-1998-Metrics-on-state-from-action-of-cmpt-gp}. This is because after working in the function space one needs to go back to the $C^{\ast}$-algebra. We conjecture that in general for $C^{\ast}_{r}(\mathcal G)$ even when $\mathcal G$ is some transformation groupoid, the inequality mentioned in \cite[Condition 1.5]{Rieffel-1998-Metrics-on-state-from-action-of-cmpt-gp} fails.
			\end{rem}
			
			{\bf Acknowledgment}: The first author acknowledges the financial support under the Junior Research Fellowship Scheme funded by UGC. The third author acknowledges the support from the SERB MATRICS grant (MTR/2022/000515).

			

			\bibliographystyle{plainnat}
			\bibliography{References}

		\end{document}